\documentclass[11pt,a4paper]{amsart}

\usepackage{tikz}
\usetikzlibrary{matrix,arrows,calc,snakes,patterns,decorations.markings}

\usepackage[mathscr]{eucal}
\usepackage{graphics,epic}
\usepackage{amsfonts}
\usepackage{amscd}
\usepackage{latexsym}
\usepackage{amsmath,amssymb, amsthm, stmaryrd, bm, bbm}
\usepackage[all,2cell]{xy}
\usepackage{mathrsfs}

\newtheorem{theorem}{Theorem}[section]
\newtheorem*{theorem*}{Theorem}

\newtheorem{lemma}[theorem]{Lemma}
\newtheorem{proposition}[theorem]{Proposition}

\newtheorem{conjecture}[theorem]{Conjecture}
\newtheorem*{conjecture*}{Conjecture}
\newtheorem{question}[theorem]{Question}

\newtheorem{remark}[theorem]{Remark}

\newcommand{\ie}{{\em i.e.~}\ }
\newcommand{\confer}{{\em cf.~}\ }

\newcommand{\opname}[1]{\operatorname{\mathsf{#1}}}

\renewcommand{\mod}{\opname{mod}\nolimits}

\newcommand{\proj}{\opname{proj}\nolimits}

\newcommand{\add}{\opname{add}\nolimits}

\newcommand{\Loc}{\opname{Loc}\nolimits}
\newcommand{\rad}{\mathrm{rad}}

\renewcommand{\ker}{\opname{ker}\nolimits}

\newcommand{\thick}{\opname{thick}\nolimits}

\newcommand{\per}{\opname{per}\nolimits}

\renewcommand{\P}{\mathbb{P}}

\newcommand{\id}{\mathrm{id}}

\newcommand{\Hom}{\opname{Hom}}
\newcommand{\End}{\opname{End}}

\newcommand{\RHom}{\opname{RHom}}
\newcommand{\cHom}{\mathcal{H}\it{om}}
\newcommand{\cEnd}{\mathcal{E}\it{nd}}

\newcommand{\ten}{\otimes}
\newcommand{\lten}{\overset{\boldmath{L}}{\ten}}

\newcommand{\Cone}{\opname{Cone}}

\newcommand{\cb}{{\mathcal B}}
\newcommand{\cc}{{\mathcal C}}
\newcommand{\cd}{{\mathcal D}}
\newcommand{\ce}{{\mathcal E}}

\newcommand{\ch}{{\mathcal H}}

\newcommand{\cm}{{\mathcal M}}

\newcommand{\cp}{{\mathcal P}}

\newcommand{\ct}{{\mathcal T}}

\renewcommand{\hat}[1]{\widehat{#1}}

\newcommand{\m}{\mathfrak{m}}
\newcommand{\del}{\partial}

\numberwithin{equation}{section}
\newtheorem{innerintrothm}{Theorem}
\newenvironment{introthm}[1]
  {\renewcommand\theinnerintrothm{#1}\innerintrothm}
  {\endinnerintrothm}

\begin{document}

\title[Relative singularity categories III]{Relative singularity categories III: Cluster resolutions}

\author{Martin Kalck}
\thanks{M.K. was supported by the DFG grant Bu--1866/2--1 and ERSRC grant EP/L017962/1.}
\address{~~
Institut f\"ur Mathematik und Wissenschaftliches Rechnen, Universit\"at Graz, 8010 Graz, Austria
}

\email{martin.kalck@uni-graz.at}
\author{Dong Yang}
\thanks{D.Y. was supported by the National Science Foundation in China No. 11401297 and the DFG program SPP 1388 (YA297/1-1 and KO1281/9-1).}
\thanks{The authors are grateful to the referee for helpful comments.}
\address{Dong Yang, School of Mathematics, Nanjing University, Nanjing 210093, P. R. China}
\email{yangdong@nju.edu.cn}

\date{\today}
\maketitle

\begin{abstract}
We build the foundation of an approach studying the canonical form of $2$-Calabi--Yau triangulated categories with cluster-tilting objects using dg algebras and relative singularity categories. This is motivated by cluster theory, singularity categories and Wemyss's Homological Minimal Model Program and the relations between these topics.
\end{abstract}

\section{Introduction}
\noindent
Triangulated categories provide a natural framework for homological algebra and their structural properties are widely studied in mathematics and theoretical physics. The most important symmetry of a triangulated category is the build-in \emph{shift functor} $\Sigma$. Using classical works on dualities of Nakayama \cite{Nakayama39, Nakayama41} and Serre \cite{Serre55}, many triangulated categories arising in representation theory and algebraic geometry admit another very important automorphism called \emph{Serre functor}, see Happel \cite{Happel88} and Bondal \& Kapranov \cite{BK89}, respectively. A triangulated category is called \emph{$d$-Calabi--Yau} if there is a simple relation between these two functors: namely, if the $d$-th power $\Sigma^d$ of the shift functor is a Serre functor.  

Famous examples include derived categories of coherent sheaves on smooth \emph{Calabi-Yau varieties} (i.e. varieties with trivial canonical bundle), the related \emph{Kuznetsov component} for Fano varieties \cite{KuznetsovICM}, Buchweitz-Orlov \emph{singularity categories} \cite{Buchweitz87, Orlov04, Orlov09} of varieties with Gorenstein (isolated) singularities \cite{Auslander76, IyamaWemyss14} and \emph{cluster categories}, providing key insights into Fomin--Zelevinsky's cluster algebras \cite{FominZelevinsky02} via categorification \cite{BuanMarshReinekeReitenTodorov06,BuanMarshReiten08,CalderoKeller08,Palu08,Plamondon11}.

The study of canonical forms for derived and triangulated categories led to the development of \emph{Tilting theory} by Baer \cite{Baer88}, Bondal \cite{Bondal89}, Happel \cite{Happel88} and Rickard \cite{Rickard89}, generalising Beilinson's seminal work on derived categories of projective spaces $\P^n$, \cite{Beilinson78}. 
However, the definition of Calabi--Yau categories, does not allow any \emph{tilting objects}. \emph{Cluster-tilting theory} 
overcomes this problem in many cases (see e.g., Keller \& Reiten's recognition theorem \cite{KellerReiten08}, which we partly recover in Theorem \ref{thm:KR's-recognition-theorem}), and also plays a key role in the study of cluster algebras and Wemyss's homological minimal model program \cite{Wemyss18}. Indeed, \emph{cluster-tilting objects} in cluster categories and in singularity categories correspond to, respectively, seeds of cluster algebras \cite{BuanMarshReinekeReitenTodorov06} and Van den Bergh's noncommutative crepant resolutions (NCCRs) \cite{VandenBergh04, Iyama07}. Given their structural similarities, it is natural to ask for equivalences between cluster categories and singularity categories -- this is identified as a main problem in Iyama's 2018 ICM talk \cite{IyamaICM}. 
Motivated by this, we study the structure of $2$-Calabi--Yau categories with cluster tilting objects, approaching the following Morita-type conjecture of Amiot \cite{Amiot08}.


\begin{conjecture}\label{conj:morita-for-2-cy-intro} 
Let $k$ be an algebraically closed field of characteristic zero. 
Let $\cc$ be a 2-Calabi--Yau algebraic triangulated category with a cluster-tilting object $T$. Then $\cc$ is triangle equivalent to the cluster category of a quiver with potential.
\end{conjecture}

\noindent
One approach to attack this conjecture is developed in \cite{Amiot09} and further pursued in \cite{AmiotReitenTodorov11,AmiotIyamaReitenTodorov12,AmiotIyamaReiten15}. An alternative strategy, using relative singularity categories, was suggested in \cite{deVolcseyVandenBergh16,KalckYang16}. We state the main result of this paper, which builds the foundation of our approach. 

\begin{introthm}{\ref{thm:cluster-resolution}}\label{T:Main}
Let $k$ be an algebraically closed field of characteristic zero. Let $\cc$ be a $d$-Calabi--Yau algebraic triangulated $k$-category with a $d$-cluster-tilting object $T$.
Then there is a dg $k$-algebra $B$ such that
\begin{itemize}
\item[(a)] $H^p(B)=0$ for $p>0$, $H^0(B)\cong \End_\cc(T)$ and $H^p(B)\cong \Hom_\cc(T,\Sigma^{p}T)$ for $p<0$;
\item[(b)] $\per(B)\supseteq\cd_{fd}(B)$;
\item[(c)] there is a triangle equivalence $\per(B)/\cd_{fd}(B)\to \cc$ which takes $B$ to $T$;
\item[(d)] there is a bifunctorial isomorphism for $X\in\mod H^0(B)$ and $Y\in\cd_{fd}(B)$
\[
D\Hom_{\cd_{fd}(B)}(X,Y)\cong \Hom_{\cd_{fd}(B)}(Y,\Sigma^{d+1}X).
\]
\end{itemize}
\end{introthm}
Known examples of categories $\cc$ satisfying the assumptions of Theorem \ref{T:Main} are singularity categories $D_{sg}(R)$ of $(d+1)$-dimensional Gorenstein isolated singularities $R$ admitting an NCCR. In order to explain the relation between our Theorem and Amiot's Conjecture \ref{conj:morita-for-2-cy-intro}, 
we first recall the definition of a \emph{cluster category} $\cc_{(Q,W)}$ of a quiver $Q$ with potential $W$:
\begin{align}
\cc_{(Q,W)} = \per(B)/\cd_{fd}(B),
\end{align}
where $B=\widehat{\Gamma}(Q,W)$ is the \emph{complete Ginzburg dg algebra} associated to $Q$ and $W$, \cite{Ginzburg06}. 

In particular, taking $d=2$ in our Theorem reduces Conjecture \ref{conj:morita-for-2-cy-intro} to the problem of finding quasi-equivalences between the dg algebras $B$ appearing in Theorem \ref{T:Main} and \emph{complete Ginzburg dg algebras} $\widehat{\Gamma}(Q,W)$. To show the latter, 
it is sufficient (\cite[Theorem B]{VandenBergh15}) to show that the dg algebras $B$ are quasi-equivalent to \emph{exact} $3$-bimodule dg algebras in the sense of \cite[Section 1, Definition]{VandenBergh15}, leading to Question \ref{Q:CYpropDG}. The problem of comparing different Calabi--Yau properties of dg algebras is also studied in \cite{Bocklandt08,Ginzburg06,VandenBergh15}.

\begin{remark}
The dg algebra $B$ in Theorem~\ref{thm:cluster-resolution} is a dg universal localisation of a certain algebra (possibly with several objects) $A$ \cite{KalckYang16,KalckYang18a,Booth20}. If $A$ is quasi-isomorphisc to a complete Ginzburg dg algebra, so is $B$. For example, let $R$ be a complete local Gorenstein isolcated singularity with residue field $k$ of dimension $3$ admitting an NCCR $A$. Then 
by Van den Bergh's \cite{VandenBergh15}, $A$ is quasi-isomorphic to a complete Ginzburg dg algebra. It follows that Amiot's conjecture holds for $\cc=\cd_{sg}(R)$,
see \cite{deVolcseyVandenBergh16}. And see \cite{HuaKeller18, Booth20} for applications to a derived version of the Donovan-Wemyss conjecture \cite{DonovanWemyss16}, where $R$ is a hypersurface singularity. 

Moreover, Theorem \ref{T:Main} (a) was also observed by Booth \cite[Theorem 6.4.2.]{Booth20} in case $\cc=\cd_{sg}(R)$ for $d+1$-dimensional Gorenstein singularities $R$. In particular, if $d=2$ and $R$ is a hypersurface, then the shift functor is $2$-periodic by work of Eisenbud \cite{Eisenbud80} and Theorem \ref{T:Main} (a) implies that the cohomology of $B$ satisfies
\begin{align}
H^i(B) = \begin{cases} 0 \qquad \qquad \qquad \text{  if  } i>0, \\  0 \qquad \qquad \qquad \text{  if  } i \leq 0 \text{  is odd,  }  \\ \End_{D_{sg}(R)}(T) \quad \text{  if  } i \leq 0 \text{  is even.  }\end{cases} 
\end{align}
\end{remark}

%

We also use Theorem \ref{thm:cluster-resolution} to give an alternative proof of Keller \& Reiten's recognition theorem \cite[Theorems 2.1 and 4.2]{KellerReiten08} in the special case that the quiver $Q$ is a tree.

\begin{introthm}
{\ref{thm:KR's-recognition-theorem}}
Let $\cc$ be a $d$-Calabi--Yau algebraic triangulated category. Assume that $T$ is a $d$-cluster-tilting object of $\cc$ such that $\Hom_\cc(T,\Sigma^p T)=0$ for $-d+2\leq p\leq -1$ and that $\End_\cc(T)\cong kQ$ for some finite tree quiver $Q$.

Then there is an equivalence of triangulated categories
\[
\cc \cong \cc^{(d)}_Q,
\]
where $\cc^{(d)}_Q$ is a triangulated orbit category in the sense of Keller \cite{Keller05}:
\[
\cc^{(d)}_Q:=\cd^b(\mod kQ)/\tau^{-1}\circ \Sigma^{d-1}.
\]
\end{introthm}

\medskip

Recently, a general framework to relate singularity categories and cluster categories is established in \cite{HaniharaIyama22}, and a variant of Amiot's Conjecture \ref{conj:morita-for-2-cy-intro} together with its higher-dimensional and relative version is proved in \cite{KellerLiu23}. Here is the statement for the higher-dimensional version: Under suitable assumptions, a triangulated category $\cc$ is triangle equivalent to the cluster category associated with a (d+1)-dimensional deformed dg preprojective algebra if and only if $\cc$ contains a $d$-cluter-tilting object and admits a dg enhancement endowed with a right $d$-Calabi--Yau structure.

\smallskip
Throughout the paper, let $k$ be a commutative ring.

\section{DG algebras and $A_\infty$-algebras}

A \emph{dg $k$-algebra} $A$ is a $\mathbb{Z}$-graded $k$-algebra $A=\bigoplus_{p\in\mathbb{Z}}A^p$ endowed with a differential $d$ of degree $1$ such that the graded Leibniz rule holds
\[
d(ab)=d(a)b+(-1)^p ad(b)
\]
for all $a\in A^p$ ($p\in\mathbb{Z}$) and $b\in A$.

\subsection{Complete dg-quiver algebras}
Assume that $k$ is a field.

\smallskip
Let $Q$ be a graded quiver such that that $Q_0$ is finite. The \emph{complete graded path algebra} $\widehat{kQ}$ of $Q$ is the completion of the graded path algebra $kQ$ at the graded ideal $\m$ generated by the arrows of $Q$ in the category of graded $k$-algebras. So the degree $n$ component of $\widehat{kQ}$ consists of elements of the form $\Sigma_p \lambda_p p$, where $\lambda_p\in k$ and $p$ runs over all paths of $Q$ of degree $n$. We refer to \cite[Section II.3]{CaenepeelVanOystaeyen88} for the theory on completions of graded rings. Let $\widehat{\m}$ be the completion of $\m$ in $\widehat{kQ}$, \ie $\widehat{\m}$ consists of elements of the form $\Sigma_p \lambda_p p$, where $\lambda_p\in k$ and $p$ runs over all non-trivial paths of $Q$. Then the $\widehat{\m}$-adic topology on $\widehat{kQ}$ is pseudo-compact (\cite{Keller11a,VandenBergh15}). For an ideal $I$ of $\widehat{kQ}$ we denote by $\overline{I}$ the closure of $I$ with respect to this topology.

Let $r\in\mathbb{N}$ and $K$ be the direct product of $r$ copies of $k$ (with standard basis $e_1,\ldots,e_r$). Let $V$ be a graded $K$-$K$-bimodule. Thencomplete tensor algebra $\widehat{T}_K V:=\prod_{p\geq 0}V^{\ten_K p}$ is isomorphic to the complete graded path algebra $\widehat{kQ}$ of the graded quiver $Q$ which has vertex set $\{1,\ldots,r\}$ and which has $\dim_k e_jV^m e_i$ arrows of degree $m$ from $i$ to $j$.

\smallskip
We call a dg $k$-algebra a \emph{complete dg-quiver algebra}  if it is of the form $A=(\widehat{kQ},d)$, where $Q$ is a
graded quiver with finitely many vertices and
 $d\colon \widehat{kQ}\rightarrow \widehat{kQ}$ is a continuous $k$-linear
differential of degree $1$ such that $d$ takes all
trivial paths to $0$.  It is \emph{minimal} if $d$ takes
an arrow of $Q$ into $\widehat{\m}^2$. By the graded Leibniz rule, the differential $d$ is determined by its value on arrows.  Since $d$ takes all trivial paths to $0$, it follows (again by the graded Leibniz rule) that $d$ takes an arrow $\alpha$ to a (possibly infinite) linear combination of paths with source $s(\alpha)$ and target $t(\alpha)$.

\begin{lemma}
\label{lem:dg-quiver-algebra-with-hereditary-0-th-cohomology}
Let $A=(\widehat{kQ},d)$ be a minimal complete dg-quiver algebra such that $Q$ is concentrated in non-positive degrees and let $m$ be a positive integer. If $H^{p}(C)=0$ for all $-m\leq p\leq -1$ and if $H^0(A)$ is hereditary, then $Q$ has no arrows in degree $p$ for any $-m\leq p\leq -1$.
\end{lemma}
\begin{proof}
Let $Q_1^p$ denote the set of arrows of $Q$ of degree $p$. Since $H^0(A)$ is hereditary, we have $d|_{Q_1^{-1}}=0$. We prove the statement by induction on $m$. Any arrow of degree $-1$ is contained in the kernel but not in the image  of $d$. Thus $H^{-1}(A)=0$ implies that $Q_1^{-1}$ is empty, so the statement holds true for $m=1$. Now assume $m\geq 2$. By induction hypothesis $Q_1^p$ is empty for $-m+1\leq p\leq -1$. Thus any arrow of degree $-m$ is contained in the kernel but not in the image of $d$. Thus $H^{-m}(A)=0$ implies that $Q_1^{-m}$ is empty.
\end{proof}

\subsection{A Duality}\label{ss:duality} 
This subsection generalises part of \cite[Section 10]{Keller94}. Assume that $k$ is a field.
\smallskip

Let $A$ be a dg $k$-algebra and $M$ be a dg $A$-module. We assume that $M$ is $\ch$-projective or $\ch$-injective and put $B=\cEnd_A(M)$. Then $M$ becomes a dg $B$-$A$-bimodule and we have an adjoint pair of triangle functors
\[
\begin{xy}
\SelectTips{cm}{}
\xymatrix{\cd(B)\ar@<.7ex>[rr]^{?\lten_B M}&&\cd(A)\ar@<.7ex>[ll]^{\RHom_A(M,?)}.}
\end{xy}
\]
Let $\tilde{M}$ be an $\ch$-projective resolution of $M$ over $B^{op}\ten A$. Then $\tilde{M}$ is $\ch$-projective over both $B^{op}$ and $A$. So we have isomorphisms of triangle functors
\[?\lten_B M\cong ?\lten_B\tilde{M}\cong ?\ten_B \tilde{M}\]
and 
\[\RHom_A(M,?)\cong \RHom_A(\tilde{M},?)\cong \cHom_A(\tilde{M},?).\]

Let $M^*=\cHom_A(\tilde{M},D(A))=D(\tilde{M})$. Then $M^*$ is a dg $\cEnd_A(D(A))$-$B$-bimodule. Let $N$ be an $\ch$-projective or $\ch$-injective resolution of $M^*$ over $B$ and put $C=\cEnd_B(N)$. Let $\tilde{N}$ be an $\ch$-projective resolution of $N$ over $C^{op}\ten B$. Consider the dg functor
\[
F=\cHom_A(\tilde{M},?)\circ (?\ten_{\cEnd_A(D(A))} D(A))\colon  \cc_{dg}(\cEnd_A(D(A)))\to\cc_{dg}(B),
\]
which takes $\cEnd_A(D(A))$ to $M^*$ and whose left derived functor is
\[
\mathbb{L}F=\cHom_A(\tilde{M},?)\circ (?\lten_{\cEnd_A(D(A))} D(A))\colon  \cd(\cEnd_A(D(A)))\to\cd(B).
\]
By \cite[Lemma 7.3(a)]{Keller94}, $Y=\cHom_B(\tilde{N},M^*)=\cHom_B(\tilde{N},F(\cEnd_A(D(A))))$ is a quasi-functor from $\cEnd_A(D(A))$ to $C$, and it is a quasi-equivalence if and only if $\mathbb{L}F$ restricts to a triangle equivalence $\per(\cEnd_A(D(A)))\to\thick_{\cd(B)}(M^*)$, equivalently, $\cHom_A(\tilde{M},?)$ restricts to a triangle equivalence $\thick_{\cd(A)}(D(A))\to\thick_{\cd(B)}(M^*)$ because $?\lten_{\cEnd_A(D(A))} D(A)\colon \per(\cEnd_A(D(A)))\to \thick_{\cd(A)}(D(A))$ is always a triangle equivalence.

\begin{lemma}\label{l:duality}
If $M\in\per(A)$ and $D(A)\in\Loc_{\cd(A)}(M)$, then the quasi-functor $Y$ above is a quasi-equivalence.
\end{lemma}
\begin{proof}
If $M\in\per(A)$, then $\cHom_A(\tilde{M},?)\cong\RHom_A(M,?)$ restricts to a triangle equivalence $\Loc_{\cd(A)}(M)\to \cd(B)$. If in addition $D(A)\in\Loc_{\cd(A)}(M)$, it restricts further to a triangle equivalence $\thick_{\cd(A)}(D(A))\to \thick_{\cd(B)}(M^*)$.
\end{proof}

\subsection{Koszul duality}\label{ss:koszul-duality}
Assume that $k$ is field. 
Fix $r\in\mathbb{N}$. Let $K$ be the direct product of $r$ copies of $k$ and consider it as a $k$-algebra via the diagonal embedding.

A dg algebra over $K$ is a dg $k$-algebra $A$ together with a dg $k$-algebra homomorphism $\eta\colon K\to A$, called the \emph{unit}. It is \emph{augmented} if in addition there is a dg $k$-algebra homomorphism $\varepsilon\colon A\to K$, called the \emph{augmentation map}, such that $\varepsilon\eta=\id_K$.

Let $A$ be an augmented dg algebra over $K$. Denote by $\bar{A}=\ker\varepsilon$. Note that $\bar{A}$ is a dg ideal of $A$.
The \emph{bar construction} of $A$, denoted by $BA$, is the graded $K$-bimodule
$$T_{K}(\bar{A}[1])=K\oplus \bar{A}[1]\oplus\bar{A}[1]\ten_{K}\bar{A}[1]\oplus\ldots.$$
It is naturally a coalgebra with comultiplication $\Delta\colon  BA\to BA\ten_K BA$ defined by
splitting the tensors:
\[
\Delta(a_1\ten\cdots \ten a_n)=\sum_{p=0}^n (a_1\ten\cdots\ten a_p)\ten (a_{p+1}\ten\cdots\ten a_n).
\]
Moreover, the differential and multiplication on $\bar{A}$ uniquely extend
to a $K$-bilinear differential on $BA$, making it a dg coalgebra over $K$:
\begin{align*}
d_{BA}(a_1\ten \cdots\ten a_n)&=\sum_{p=1}^{n-1}(-1)^{|a_1|+\ldots+|a_p|+1} a_1\ten\cdots\ten a_p\ten d_A(a_{p+1})\ten a_{p+2}\ten\cdots\ten a_n\\
&+\sum_{p=1}^{n-2}(-1)^{|a_1|+\ldots+|a_p|+|a_{p+1}|}a_1\ten\cdots\ten a_p\ten a_{p+1}a_{p+2}\ten a_{p+3}\ten\cdots\ten a_n,
\end{align*}
where $a_1,\ldots,a_n$ are homogeneous elements of $\bar{A}[1]$ of degree $|a_1|,\ldots,|a_n|$, respectively.
The dual bar construction of $A$ is the graded $k$-dual of $BA$:
\[E(A)=B^{\#}A:=D(BA).\]
As a graded algebra $E(A)=\hat{T}_{K}(D(\bar{A}[1]))$ is the complete tensor algebra of $D(\bar{A}[1])=\Hom_k(\bar{A}[1],k)$
over $K$. It is naturally an augmented dg algebra over $K$ with differential $d$ being the unique continuous $K$-bilinear map satisfying the graded Leibniz rule and taking 
$f \!\in D(\bar{A}[1])$ to $d(f)\! \in D(\bar{A}[1])\oplus D(\bar{A}[1]\ten_K \bar{A}[1])$, defined by
\begin{align*}
d(f)(a_1)&=-f(d_A(a_1)),\\
d(f)(a_1\ten a_2)&= (-1)^{|a_1|} f(a_1a_2),
\end{align*}
where $a_1,a_2$ are homogeneous elements of $\bar{A}[1]$ of degree $|a_1|, |a_2|$, respectively.

There is a certain $\ch$-projective resolution $M$, which we call the \emph{bar resolution}, of $K$ (viewed as a dg $A$-module via the augmentation map), such that $\cEnd_A(M)=B^{\#}A$. See \cite{SuHao18} and \cite[Section 19 exercise 4]{FelixHalperinThomas01}.

\subsection{$A_\infty$-algebras}\label{ss:A-infinity-algebra}

Assume that $k$ is a field and fix $r\in\mathbb{N}$. Let $K$ be the direct product of  $r$ copies of $k$ and consider it as a $k$-algebra via the diagonal embedding.

An \emph{$A_\infty$-algebra} $A$ over $K$ is a graded $K$-bimodule endowed with a family of homogenous $K$-bilinear maps $\{m_n\colon A^{\ten_K
n}\rightarrow A |n\geq 1\}$ of degree
$2-n$ (called \emph{multiplications})  satisfying certain conditions.  We need the following facts, see \cite{Keller06c,Lefevre03,LuPalmieriWuZhang04,LuPalmieriWuZhang08,KalckYang16}.
\begin{itemize}
\item[(a)] A dg algebra over $K$ is a strictly unital $A_\infty$-algebra over $K$ with $m_1$ being the differential, $m_2$ being the multiplication and $m_n=0$ for $n\geq 3$.
\item[(b)] (\cite[Corollarie 3.2.4.1]{Lefevre03}) A strictly unital $A_\infty$-algebra $A$ over $K$ admits a minimal model, that is, a strictly unital minimal $A_\infty$-algebra  which is $A_\infty$-quasi-isomorphic to $A$.
\item[(c)] (\cite[Lemme 2.3.4.3]{Lefevre03}) Let $A$ be an augmented $A_\infty$-algebra over $K$. Then there is an augmented dg algebra $U(A)$ over $K$, called the \emph{enveloping dg algebra of $A$}, together with an $A_\infty$-quasi-isomorphism $\psi\colon  A\to U(A)$ of augmented $A_\infty$-algebras over $K$. It satisfies the following universal property: if $B$ is a dg algebra over $K$ and $f\colon  A\to B$ is a strictly unital $A_\infty$-morphism of strictly unital $A_\infty$-algebras over $K$, then there is a unique homomorphism $f'\colon U(A)\to B$ of dg algebras over $K$ such that $f=f'\circ\psi$.
\item[(d)] (\cite[Section 2.7]{KalckYang16}) For an augmented $A_\infty$-algebra $A$ over $K$, one can define the dual bar construction $E(A)$, which is a dg algebra over $K$. If $A$ is an augmented dg algebra over $K$, then it is exactly the one defined in in Section~\ref{ss:koszul-duality}. If $A$ and $B$ are $A_\infty$-quasi-isomorphic augmented $A_\infty$-algebras over $K$, then $E(A)$ and $E(B)$ are quasi-isomorphic as dg algebras over $K$.
\item[(e)] (\cite[Section 4.2]{SuYang19}) A strictly unital minimal $A_\infty$-algebra $A$ is said to be \emph{positive}  if
\begin{itemize}
\item[--] $A^p=0$ for all $p<0$,
\item[--] $A^0=K$ and the unit is the embedding $K=A^0\hookrightarrow A$.
\end{itemize}

\end{itemize}

\subsection{Non-positive dg algebra: augmentation}\label{ss:augmentation}
Assume that $k$ is a field. Let $A$ be a non-positive dg $k$-algebra satisfying the following three conditions:
\begin{itemize}
\item[(FD)] $H^0(A)$ is finite-dimensional over $k$,
\item[(BE)] $H^0(A)$ is elementary, \ie $H^0(A)$ is isomorphic to the quotient of the path algebra of a finite quiver by an admissible ideal,
\item[(HS)] $\per(A)\supseteq \cd_{fd}(A)$.
\end{itemize}
In this subsection we will show that $A$ is quasi-equivalent to the dual bar construction of the $A_\infty$-Koszul dual of $A$.

Let $K=H^0(A)/\rad H^0(A)$ and $r=\dim_k K$. As an algebra, $K$ is the direct product of $r$ copies of $k$. View $K$ as a dg $A$-module via the projection $A\to H^0(A)\to K$ and denote this dg module by $S$. We have a decomposition $S=S_1\oplus\ldots\oplus S_r$ such that $S_1,\ldots,S_r$ form a complete set of pairwise non-isomorphic simple $H^0(A)$-modules. 

Take (arbitrary) $\ch$-projective resolutions $X_1,\ldots,X_r$ of $S_1,\ldots,S_r$ over $A$. Put $X=X_1\oplus\ldots\oplus X_r$ and $B=\cEnd_A(X)$. Then $B$ is a dg algebra over $K$, with $K$ being identified with $k\{\id_{X_1}\}\times\cdots\times k\{\id_{X_r}\}$. Denote by $A^*$ the minimal model of $B$ (Section~\ref{ss:A-infinity-algebra}(b)) and call it the \emph{$A_\infty$-Koszul dual} of $A$. By definition $A^*$ is a strictly unital minimal $A_\infty$-algebra over $K$ and there is an $A_\infty$-quasi-isomorphism $\varphi:A^*\to B$ of strictly unital $A_\infty$-algebras over $K$. As a graded algebra over $K$, $A^*$ is the graded endomorphism algebra $\bigoplus_{p\in\mathbb{Z}}\Hom_{\cd(A)}(S,\Sigma^p S)$ of $S$. It follows from \cite[Theorem A.1(c)]{BruestleYang13} that $A^*$  is positive. In particular, it is augmented with augmentation map the projection $A^*\to (A^*)^0=K$.

\begin{theorem}
\label{t:augmentation-of-non-pos-dg-alg}
$A$ is quasi-equivalent to $E(A^*)$.
\end{theorem}
\begin{proof}
Let $U$ be the enveloping dg algebra of $A^*$ and $\psi:A^*\to U$ be the associated quasi-isomorphism of augmented $A_\infty$-algebras over $K$ (Section~\ref{ss:A-infinity-algebra}(c)). Then there is a quasi-isomorphism $E(A^*)\to E(U)$ of dg algebras over $K$ (Section~\ref{ss:A-infinity-algebra}(d)). Consider $K$ as a dg $U$-module via the augmentation map and let $Z$ be the bar resolution of $K$ over $U$ (Section~\ref{ss:koszul-duality}). Then $E(U)=\cEnd_U(Z)$.

By the universal property of $U$, there is a quasi-isomorphism $f:U\to B$ of dg algebras over $K$ such that $f\circ \psi=\varphi$. It induces a quasi-isomorphism $E(U)=\cEnd_U(Z)\to \cEnd_B(Z\ten_U B)$.

Let $\tilde{X}$ be an $\ch$-projective resolution of $X$ over $B^{op}\ten A$. Put $S^*=\cHom_A(\tilde{X},D(A))=D(\tilde{X})$, whose total cohomology is concentrated in degree $0$ and is isomorphic to $H^0(B)$ as a graded module over $H^*(B)$. So as a dg module over $U$ via the quasi-isomorphism $f:U\to B$, $S^*$ is quasi-isomorphic to $K$. As a consequence, $Z\ten_U B$ is an $\ch$-projective resolution of $S^*$ over $B$ and there is a quasi-isomorphism $\cEnd_U(Z)\to\cEnd_B(Z\ten_U B)$ of dg algebras. We claim that $D(A)\in\Loc_{\cd(A)}(S)$. Since $S\in\cd_{fd}(A)\subseteq \per(A)$, it follows from Lemma~\ref{l:duality} that $\cEnd_A(D(A))$ is quasi-equivalent to $\cEnd_B(Z\ten_U B)$, and hence to $E(A^*)$. Further, as $A$ has finite-dimensional cohomology in each degree by \cite[Proposition~2.5]{KalckYang16}, $A$ is quasi-isomorphic to $\cEnd_A(D(A))=DD(A)$. Therefore, $A$ is quasi-equivalent to $E(A^*)$.

To prove the claim, consider the chain of dg submodules
\[
\sigma^{\leq 0} D(A)\longrightarrow \sigma^{\leq 1} D(A)\longrightarrow\sigma^{\leq 2}D(A)\longrightarrow\ldots\longrightarrow D(A).
\]
We have $D(A)=\bigcup_{p\geq 0} \sigma^{\leq p}D(A)$, so $D(A)$ is the homotopy colimit of $\sigma^{\leq p}D(A)$, \ie there is a triangle
\[
\bigoplus_{p\geq 0}\sigma^{\leq p}D(A)\stackrel{\id-\textrm{shift}}{\longrightarrow}\bigoplus_{p\geq 0}\sigma^{\leq p}D(A)\longrightarrow D(A)\longrightarrow \Sigma \bigoplus_{p\geq 0}\sigma^{\leq p}D(A).
\]
Since $\sigma^{\leq p}D(A)$ belongs to $\cd_{fd}(A)=\thick_{\cd(A)}(S)$ (the equality is a consequence of \cite[Proposition~2.1(b)]{KalckYang16}), it follows that $D(A)\in\Loc_{\cd(A)}(S)$.
\end{proof}

We remark that $E(A^*)$ has the form $(\widehat{kQ},d)$, which is a complete dg-quiver algebra. The graded quiver $Q$ is determined by the graded $K$-bimodule structure on $A^*=\bigoplus_{p\in\mathbb{Z}}\Hom_{\cd(A)}(S,\Sigma^p S)$. More precisely, $Q_0=\{1,\ldots,r\}$ and the number of arrows from $i$ to $j$ in degree $p$ is the dimension of $\Hom_{\cd(A)}(S_j,\Sigma^{1-p}S_i)$ over $K$. In particular, $Q$ is finite and concentrated in non-positive degrees. The differential $d$ is continuous and is determined by the $A_\infty$-structure on $A^*$.

\subsection{Trivial extensions}
Assume that $k$ is a field.
Let $Q$ be a finite tree quiver, and let $A=kQ/\rad^2 kQ$. For an $A$-bimodule $M$, define the trivial extension $A\ltimes M$ of $A$ by $M$ as follows. As a vector space it is $A\oplus M$. The multiplication is given by
\[
(a,m)(a',m')=(aa',am'+ma').
\]

Let $\lambda,\mu\colon Q_1\to k^{\times}$ be two functions. 
For an $A$-bimodule $M$, define a new bimodule ${}^\lambda M^\mu$ by
\[
\alpha\cdot m=\lambda(\alpha)\alpha m,~~m\cdot\alpha=\mu(\alpha)m\alpha,~~\text{where }m\in M,\alpha\in Q_1.
\]
Put $A(Q,\lambda,\mu)=A\ltimes {}^\lambda D(A)^\mu$. Precisely, as a $k$-vector space $A(Q,\lambda,\mu)$ has basis
\[
e_i,~e_i^* (i\in Q_0), ~~\alpha,~\alpha^* (\alpha\in Q_1), 
\]
and the product of two basis elements are zero except for $i\in Q_0$
\begin{align*}
e_i^2=e_i, ~~ e_ie_i^*=e_i^*e_i=e_i^*, 
\end{align*}
anf for $\alpha\in Q_1$
\begin{align*}
&e_{t(\alpha)}\alpha=\alpha e_{s(\alpha)}=\alpha, ~~e_{s(\alpha)}\alpha^*=\alpha^*e_{t(\alpha)}=\alpha^*,\\
&\alpha^*\alpha=\mu(\alpha)e_{s(\alpha)}^*, \alpha\alpha^*= \lambda(\alpha)e_{t(\alpha)}^*.
\end{align*}
For $d\geq 2$, let $A_d(Q,\lambda,\mu)$ be the graded algebra whose underlying algebra is $A(Q,\lambda,\mu)$ and whose grading is given by $|e_i|=0$ and $|e_i^*|=d+1$ for $i\in Q_0$ and $|\alpha|=1$ and $|\alpha^*|=d$ for $\alpha\in Q_1$.

\begin{lemma}
\label{lem:formality-of-trivial-extension}
Let $Q$ be a finite tree quiver, $d\geq 2$ and $\lambda,\mu\colon Q_1\to k^\times$ be two functions. Let $B$ be a strictly unital minimal $A_\infty$-algebra over $K=kQ_0$ such that its graded algebra $(B,m_2)$ is isomorphic to $A_d(Q,\lambda,\mu)$. Then $m_n=0$ for $n\geq 3$.
\end{lemma}
\begin{proof}
We first show that $m_n|_{(B^1)\ten_K n}\colon ((B^1)^{\ten_K n} \to (B^*)^2$ is trivial. This is obvious for $d\geq 3$ because $B^2=0$. For $d=2$, $m_n(\alpha_1\ten\cdots\ten\alpha_n)\neq 0$ for $\alpha_1,\ldots,\alpha_n\in Q_1$ implies that $\alpha_1\cdots\alpha_n$ is a path and there exists $\alpha\in Q_1$ such that $s(\alpha_n)=t(\alpha)$ and $t(\alpha_1)=t(\alpha)$. This contradicts the assumption that $Q$ is a tree.
Now for degree reasons and due to the form of $B^{d+1}$, the only possible non-trivial multiplication on $B^{\geq 1}$ is of the form
\[
m_n(\alpha_1\ten\cdots\ten\alpha_{s-1}\ten\alpha_s^*\ten\alpha_{s+1}\ten\cdots\ten\alpha_n),
\]
where $\alpha_{s+1}\cdots\alpha_n\alpha_1\cdots\alpha_{s-1}$ is a path parallel to $\alpha_s$. Since $Q$ is a tree quiver, this is not possible unless $n=2$ and $\alpha_1=\alpha_2$.
\end{proof}

\begin{lemma}
\label{lem:no-deformation-trivial-extension} Let $Q$ be a finite tree quiver.
Then $A(Q,\lambda,\mu)$ is isomorphic to $A(Q,\mathbf{1},\mathbf{1})$, where $\mathbf{1}\colon Q_1\to k^\times$ is the function with constant value $1$.
\end{lemma}
\begin{proof}
It is enough to show that as an $A$-bimodule ${}^\lambda D(A)^\mu$ is isomorphic to $D(A)$.

Fix a vertex $i\in Q_0$. Define $\lambda',\mu'\colon Q_1\to k^\times$ by
\[
\lambda'(\beta)=\begin{cases} \lambda(\beta) & \text{if } t(\beta)\neq i\\ 1 & \text{if } t(\beta)=i \end{cases},\qquad \mu'(\beta)=\begin{cases} \mu(\beta) & \text{if } s(\beta)\neq i\\ 1 & \text{if } s(\beta)=i \end{cases}.
\]
We claim that ${}^\lambda D(A)^\mu\cong {}^{\lambda'} D(A)^{\mu'}$ as $A$-bimodules. Then fixing an ordering $i_1\cdots i_r$ of $Q_0$ and repeatedly applying the claim we obtain the desired result.

Now we prove the claim. Because $Q$ is a tree, there is at most one walk from any vertex to $i$. For $j\in Q_0$, put 
\[
f(j)=\begin{cases} \lambda(\alpha) & \text{if there is a walk from $j$ to $i$ ending with an arrow $\alpha$},\\
\mu(\alpha) & \text{if there is a walk from $j$ to $i$ ending with the inverse of an arrow $\alpha$},\\
1 & \text{otherwise.}
\end{cases}
\]
For $\beta\in Q_1$, put
\[
f(\beta)=\begin{cases} f(t(\beta)) & \text{if } t(\beta)\neq i,\\
f(s(\beta)) & \text{if } s(\beta)\neq i,\\
\mu(\beta) & \text{if } s(\beta)=i,\\
\lambda(\beta) & \text{if } t(\beta)=i.
\end{cases}
\]
Let $\{e_j^*|j\in Q_0\}\cup \{\beta^*|\beta\in Q_1\}\subseteq D(A)$ be the dual basis of $\{e_i|i\in Q_0\}\cup\{\beta|\beta\in Q_1\}$. Let $\varphi\colon {}^\lambda D(A)^\mu\cong {}^{\lambda'} D(A)^{\mu'}$ be the linear extension of $\beta^*\mapsto f(\beta)\beta^*$ and $e_j^*\mapsto f(j)e_j^*$. It is straightforward to check that $\varphi$ is an $A$-bimodule isomorphism.
\end{proof}

\section{Silting reduction}
\label{s:silting-reduction}

\newcommand{\emp}{\cd}
Let $\ce$ be a Frobenius $k$-category and $\cp$ the full subcategory of projective-injective objects of $\ce$. Put $\emp=\ch^b(\ce)/\ch^b(\cp)$.  The following proposition is the main result of this section.

\begin{proposition}\label{prop:silting-reduction}
Let $\cm$ be an additive subcategory of $\ce$ containing $\cp$. Then the essential image of the composite functor $\cm\to\ch^b(\cm)\to\ch^b(\cm)/\ch^b(\cp)$ is a silting subcategory and is equivalent to the additive quotient $\frac{\cm}{\cp}$.
\end{proposition}

Proposition~\ref{prop:silting-reduction} is an immediate consequence of the following lemma, which also allows us to compute the extension spaces in negative degrees of objects of $\cm$ in $\ch^b(\cm)/\ch^b(\cp)$ in terms of morphism spaces in $\underline{\ce}$.

\begin{lemma}\label{lem:silting-reduction}
For $X,Y\in\ce$, we have
\begin{align*}
\Hom_{\emp}(X,\Sigma^n Y)&=0 \text{ for } n>0,\\
\Hom_{\emp}(X, Y)&=\Hom_{\underline{\ce}}(X,Y),\\
\Hom_{\emp}(X,\Sigma^{n} Y)&=\Hom_{\underline{\ce}}(X,\Omega^{-n}(Y)) \text{ for } n<0.
\end{align*}
\end{lemma}
\begin{proof} The first formula will be proved in Step 2. The second and third formulas will be proved in Step 5.

Step 1: \emph{In a left fraction $X\stackrel{s}{\leftarrow} Z\stackrel{f}{\rightarrow} \Sigma^n Y$, where $\Cone(s)\in\ch^b(\cp)$, up to equivalence we can take $Z$ to be a complex of the form
\[
X\stackrel{\alpha}{\to} P^0\stackrel{d^0}{\to} P^1\to\ldots\stackrel{d^{l-1}}{\to} P^l,
\]
where $P^i\in\cp$ and $X$ is put in degree $0$, and take $s$ to be the natural projection from $Z$ to $X$.} 

Form a triangle
\[
Z\stackrel{s}{\to} X\stackrel{u}{\to} P\to \Sigma Z,
\]
where $P\in\ch^b(\cp)$ is of the form
\[
P^{m}\stackrel{d^m}{\to} \ldots\to P^{-1}\stackrel{d^{-1}}{\to} P^0\stackrel{d^0}{\to} P^1\to\ldots \to P^l.
\]
The morphism $u\colon X\to P$ is represented by a chain map $X\to P$, which is given by a morphism $\alpha\colon X\to P^0$ in $\ce$ such that $d^0\circ\alpha=0$:
\[
\xymatrix{
& & & X\ar[d]^{\alpha} &&&\\
P^m\ar[r]^{d^m} &\ldots \ar[r] & P^{-1}\ar[r]^{d^{-1}} & P^0\ar[r]^{d^0} & P^1\ar[r] &\ldots\ar[r] & P^l  
}
\]
So up to isomorphism $Z=\Sigma^{-1}\Cone(u)$ takes the form
\[
\xymatrix{
P^m\ar[r]^{d^m} &\ldots \ar[r]^(0.4){{0\choose d^{-2}}} & X\oplus P^{-1}\ar[r]^(0.6){(\alpha,d^{-1})} & P^0\ar[r]^{d^0} & P^1\ar[r] &\ldots\ar[r] & P^l  
}
\]
and $s$ is of the form
\[
\xymatrix{
P^m\ar[r]^{d^m} &\ldots \ar[r]^(0.4){{0\choose d^{-2}}} & X\oplus P^{-1}\ar[d]^{(\id_X,0)}\ar[r]^(0.6){(\alpha,d^{-1})} & P^0\ar[r]^{d^0} & P^1\ar[r] &\ldots\ar[r] & P^l  \\
& & X & & & &
}
\]
The complex $Z'$
\[
\xymatrix{
X\ar[r]^{\alpha} & P^0\ar[r]^{d^0} & P^1\ar[r] &\ldots\ar[r] & P^l  
}
\]
is a subcomplex of $Z$. The inclusion $Z'\hookrightarrow Z$ splits in each degree and the quotient complex lies in $\ch^b(\cp)$. Let $t\colon Z'\to X$ be the natural projection and $g\colon Z'\hookrightarrow Z\stackrel{f}{\to} \Sigma^n Y$. Then $X\stackrel{s}{\leftarrow} Z\stackrel{f}{\to} \Sigma^n Y$ is equivalent to $X\stackrel{t}{\leftarrow} Z'\stackrel{g}{\to} \Sigma^n Y$.

Step 2: Let $X\stackrel{s}{\leftarrow} Z\stackrel{f}{\to} \Sigma^n Y$ be as in the statement of Step 1. If $n>0$, then $\Hom_{\ch^b(\ce)}(Z,\Sigma^n Y)=0$, and hence $\Hom_{\emp}(X,\Sigma^n Y)=0$.

Step 3: \emph{Let $n\leq 0$. In a left fraction $X\stackrel{s}{\leftarrow} Z\stackrel{f}{\rightarrow} \Sigma^n Y$, where $\Cone(s)\in\ch^b(\cp)$, up to equivalence we can take $Z$ to be a complex of the form
\[
X\stackrel{\alpha}{\to} P^0\stackrel{d^0}{\to} P^1\to\ldots\stackrel{d^{-n-2}}{\to} P^{-n-1},
\]
where $P^i\in\cp$ and $X$ is put in degree $0$, and take $s$ to be the natural projection from $Z$ to $X$.}

Let $X\stackrel{s}{\leftarrow} Z\stackrel{f}{\rightarrow} \Sigma^n Y$ be as in the statement of Step 1. Let $Z''$ be the complex
\[
X\stackrel{\alpha}{\to} P^0\stackrel{d^0}{\to} P^1\to\ldots\stackrel{d^{-n-2}}{\to} P^{-n-1}.
\]
Then the quotient map $Z\to Z''$ splits in each degree and its kernel is in $\ch^b(\cp)$. Moreover, any chain map $g\colon Z\to \Sigma^n Y$ is given by a morphism $\beta\colon P^{-n-1}\to Y$ in $\ce$ such that $\beta\circ d^{-n-2}=0$:
\[
\xymatrix{
X\ar[r]^{\alpha} & P^0\ar[r]^{d^0} & P^1\ar[r] &\ldots\ar[r]^{d^{-n-2}} & P^{-n-1}\ar[r]\ar[d]^\beta & \ldots \ar[r]^{d^{l-1}} & P^l\\
& & & & Y & &  
}
\]
and factors through the quotient map $Z\to Z''$ to yield a chain map $g\colon Z''\to Y$
\[
\xymatrix{
X\ar[r]^{\alpha}\ar[d]^{\id} & P^0\ar[r]^{d^0} \ar[d]^{\id}& P^1\ar[r] \ar[d]^{\id}&\ldots\ar[r]^{d^{-n-2}} & P^{-n-1}\ar[r]\ar[d]^{\id} & \ldots \ar[r]^{d^{l-1}} & P^l\\
X\ar[r]^{\alpha} & P^0\ar[r]^{d^0} & P^1\ar[r] &\ldots\ar[r]^{d^{-n-2}} & P^{-n-1}\ar[d]^\beta\\
& & & & Y & &  
}
\]
Let $t\colon Z''\to X$ be the natural projection. Then $X\stackrel{s}{\leftarrow} Z\stackrel{f}{\rightarrow} \Sigma^n Y$ is equivalent to $X\stackrel{t}{\leftarrow} Z''\stackrel{g}{\rightarrow} \Sigma^n Y$.

Step 4: Because $\ce$ has enough projective-injective objects, there is a conflation $X\stackrel{\iota^0}{\to} I_X^0\stackrel{\pi^0}{\to} K^1$ with $I_X^0\in\cp$. Put $K^0=X$. By induction, there exist objects $K^i$ ($i\in\mathbb{N}$) and conflations $K^i\stackrel{\iota^i}{\to} I_X^i\stackrel{\pi^i}{\to} K^{i+1}$ with $I_X^i\in\cp$. Put $X(0)=X$ and let $X(l)$ ($l\geq 1$) be the complex
\[
X\stackrel{\alpha_X}{\to} I_X^0\stackrel{d_X^0}{\to} I_X^1\to \ldots \stackrel{d_X^{l-2}}{\to} I_X^{l-1},
\]
where $\alpha_X=\iota^0$ and $d_X^i=\iota^{i+1}\circ \pi^i$.

\emph{Let $n\leq 0$. In a left fraction $X\stackrel{s}{\leftarrow} Z\stackrel{f}{\rightarrow} \Sigma^n Y$, where $\Cone(s)\in\ch^b(\cp)$, up to equivalence we can take $Z=X(-n)$ and take $s$ to be the natural projection from $Z$ to $X$.}

Let $X\stackrel{s}{\leftarrow} Z\stackrel{f}{\rightarrow} \Sigma^n Y$ be as in the statement of Step 3. Then there is a chain map
\[
\xymatrix{
X\ar[r]^{\iota^0}\ar[d]^{\id} & I_X^0 \ar[r]^{\pi^0} \ar[d] & K^1\ar[d] \\
X\ar[r]^{\alpha} & P^0\ar[r]^{d_0} & P^1\ar[r] &\ldots\ar[r]^{d^{-n-2}} & P^{-n-1}
}
\]
and by induction we obtain a chain map $u\colon X(-n)\to Z$
\[
\xymatrix{
X\ar[r]^{\alpha_X}\ar[d]^{\id} & I_X^0 \ar[r]^{d_X^0} \ar[d] & I_X^1\ar[d]\ar[r] & \ldots \ar[r]^{d_X^{-n-2}} & I_X^{-n-1}\ar[d] \\
X\ar[r]^{\alpha} & P^0\ar[r]^{d_0} & P^1\ar[r] &\ldots\ar[r]^{d^{-n-2}} & P^{-n-1}
}
\]
Let $t\colon X(-n)\to X$ be the natural projection and let $g=f\circ u$. Then $X\stackrel{s}{\leftarrow} Z\stackrel{f}{\rightarrow} \Sigma^n Y$ is equivalent to $X\stackrel{t}{\leftarrow} X(-n)\stackrel{g}{\rightarrow} \Sigma^n Y$.

Step 5: Let $n\geq 0$. A chain map $X(-n)\to \Sigma^n Y$ is given by a morphism $\beta\colon I_X^{-n-1}\to Y$ in $\ce$ such that $\beta\circ d_X^{-n-2}=0$:
\[
\xymatrix{
X\ar[r]^{\alpha_X} & I_X^0 \ar[r]^{d_X^0}  & I_X^1\ar[r] & \ldots \ar[r]^{d_X^{-n-2}} & I_X^{-n-1}\ar[d]^\beta \\
& & & & Y
}
\]
and thus it is in bijection with morphisms $\gamma\colon K^{-n}\to Y$. Therefore by Step 4, there is a surjective map
\[
\Hom_\ce(K^{-n},Y)\longrightarrow \Hom_\emp(X,\Sigma^n Y).
\]
We claim that the kernel consists of the morphisms factoring though a projective-injective, which completes the proof.

Let $\gamma\colon K^{-n}\to Y$ be in the kernel of the above map, and $X\stackrel{s}{\leftarrow} X(-n) \stackrel{f}{\rightarrow} \Sigma^n Y$ be the corresponding left fraction. Then there exists a chain $t\colon Z\to X(-n)$ such that $f\circ t=0$, \ie $\gamma\circ\pi^{-n-1}\circ t^{-n}=0$.
Following the constructions in Steps 1, 3 and 4, we obtain a diagram of chain maps
\[
\xymatrix{
& Z'\ar[dl]\ar[dr] & & X(-n)\ar[dl]_{u} \\
Z\ar[dr]^t & & Z''\ar[dl]_{t'}\\
&X(-n)&
}
\]
such that the square is commutative and $t'^0\circ u^0=\id_X$. So  there is a chain map
\[
\xymatrix{
X\ar[r]^{\alpha_X}\ar[d]^{\id} & I_X^0 \ar[r]^{d_X^0} \ar[d]^{t^1\circ u^1} & I_X^1\ar[r]\ar[d]^{t^2\circ u^2} & \ldots \ar[r]^{d_X^{-n-2}} & I_X^{-n-1}\ar[d]^{t^{-n}\circ u^{-n}}\\
X\ar[r]^{\alpha_X} & I_X^0 \ar[r]^{d_X^0}  & I_X^1\ar[r] & \ldots \ar[r]^{d_X^{-n-2}} & I_X^{-n-1}
}
\]
Therefore there exist morphisms $a\colon I_X^{-n-1}\to I_X^{-n-2}$ and $b\colon I_X^{-n}\to I_X^{-n-1}$ such that $\id_{I_X^{-n-1}}-t^{-n}\circ u^{-n}=d_X^{-n-2}\circ a+b\circ d_X^{-n-1}$. So 
\begin{align*}
\gamma\circ\pi^{-n-1}&=\gamma\circ\pi^{-n-1}\circ(\id_{I_X^{-n-1}}-t^{-n}\circ u^{-n})\\
&=\gamma\circ\pi^{-n-1}\circ(d_X^{-n-2}\circ a+b\circ d_X^{-n-1})\\
&=\gamma\circ\pi^{-n-1}\circ b\circ d_X^{-n-1}\\
&=\gamma\circ \pi^{-n-1}\circ b\circ \iota^{-n}\circ\pi^{-n-1}.
\end{align*} Since $\pi^{-n-1}$ is an epimorphism, it follows that $\gamma=\gamma\circ \pi^{-n-1}\circ b\circ \iota^{-n}$, which factors through $I_X^{-n}$.
\end{proof}

\section{Relative singularity category}
\label{s:relative-singularity-category}

Let $k$ be an algebraically closed field 
 and $d\geq 1$. 
A Hom-finite Krull--Schmidt triangulated $k$-category $\cc$ is said to be \emph{$d$-Calabi--Yau} if $\Sigma^d$ is a Serre functor, that is,
there is a bifunctorial isomorphism for $M,N\in\cc$
\[D\Hom(M,N)\cong \Hom(N,\Sigma^d M).\]

Let $\ce$ be a Frobenius $k$-category, $\cp$ its full subcategory of projective-injective objects and $\cc=\underline{\ce}$ the stable category. Assume that $\cc$ is a Hom-finite Krull--Schmidt $d$-Calabi--Yau algebraic triangulated category, and $T\in\cc$ a basic $d$-cluster-tilting object. 
Then $T$ is a classical generator of $\cc$, by~\cite[Theorem 5.4 (a)]{KellerReiten07}.

\begin{theorem}\label{thm:cluster-resolution}
There is a dg $k$-algebra $B$ such that
\begin{itemize}
\item[(a)] $H^n(B)=0$ for $n>0$, $H^0(B)\cong \End_\cc(T)$ and $H^n(B)\cong \Hom_\cc(T,\Sigma^{n}T)$ for $n<0$;
\item[(b)] $\per(B)\supseteq\cd_{fd}(B)$;
\item[(c)] there is a triangle equivalence $\per(B)/\cd_{fd}(B)\to \cc$ which takes $B$ to $T$;
\item[(d)] there is a bifunctorial isomorphism for $X\in\mod H^0(B)$ and $Y\in\cd_{fd}(B)$
\[
D\Hom_{\cd_{fd}(B)}(X,Y)\cong \Hom_{\cd_{fd}(B)}(Y,\Sigma^{d+1}X).
\]
\end{itemize}
\end{theorem}
\begin{remark} 
\begin{itemize}
\item[(a)] 
If $\cp$ is skeletally small, we could prove (a), (b) and (c) by establishing several-object versions of some results in \cite{KalckYang16,KalckYang18a}. Here we use \cite{Palu09}. It is claimed in \cite{Tabuada07} that $\cd_{fd}(B)$ (equivalent to $H^0(\cb)$ there) is $(d+1)$-Calabi--Yau. This turns out to be a misunderstanding of \cite[Proposition 5.4]{KellerReiten07}.
\item[(b)] The dg algebra $B$ in Theorem~\ref{thm:cluster-resolution} is the dg quotient of $\cm$ by $\cp$. If $\cm$ has an additive generator $M$ and $\cp$ has an additive generator $P$, then there is a canonical triangle equivalence $\ch^b(\cm)/\ch^b(\cp)\to \ch^b(\proj A)/\thick(eA)$, where $A=\End(M)$, and $e=\id_P$. In this case, $B$ can be obtained using the constructions in \cite[Section 7]{KalckYang18a}.
\end{itemize}
\end{remark}

\begin{proof}
Let $\cm$ be the preimage of $\ct=\add_\cc(T)$ in $\ce$. A complex $X$ in $\ch^b(\cm)$ is said to be \emph{$\ce$-acyclic} if there are conflations $Z^i\stackrel{\iota^i}{\to} X^i\stackrel{\pi^i}{\to} Z^{i+1}$ such that $d_X^i=\iota^{i+1}\circ\pi^i$ for $i\in\mathbb{Z}$. Let $\ch^b_{\ce-ac}(\cm)$ be the full subcategory of $\ch^b(\cm)$ of $\ce$-acyclic complexes. Then $\Hom_{\ch^b(\cm)}(P,X)=0$ for $P\in\cp$ and $X\in\ch^b_{\ce-ac}(\cm)$. So $\ch^b(\cp)$ is left orthogonal to $\ch^b_{\ce-ac}(\cm)$, and we can view $\ch^b_{\ce-ac}(\cm)$ as a full subcategory of the triangle quotient $\ch^b(\cm)/\ch^b(\cp)$.

By \cite[Proposition 3]{Palu09}, there is a short exact sequence of triangulated categories
\[
\xymatrix{
0\ar[r] & \ch^b_{\ce-ac}(\cm)\ar[r] & \ch^b(\cm)/\ch^b(\cp)\ar[r] & \cc \ar[r] & 0,
}
\]
which induces a triangle equivalence $(\ch^b(\cm)/\ch^b(\cp))/\ch^b_{\ce-ac}(\cm)\to \cc$. Let $\per(\cm)$ be the full subcategory of the derived category of modules over $\cm$ classically generated by all representable functors and let $\per_{\underline{\cm}}(\cm)$ be its full subcategory consisting of complexes whose cohomologies are in $\mod \underline{\cm}$ and $\per(\cp)$ be its thick subcategory generated by $P^{\wedge},~P\in\cp$. By \cite[Lemma 7]{Palu09}, the triangle equivalence $\ch^b(\cm)\to \per(\cm)$ induces a triangle equivalence $\ch^b_{\ce-ac}(\cm)\to \per_{\underline{\cm}}(\cm)$. So there is a triangle equivalence $(\per(\cm)/\per(\cp))/\per_{\underline{\cm}}(\cm)\to \cc$.

By Proposition~\ref{prop:silting-reduction}, the representable functors form a silting subcategory of $\per(\cm)/\per(\cp)$ and is equivalent to $\underline{\cm}$. Let $\tilde{T}$ be a basic additive generator of this silting subcategory corresponding to $T$. Since $\per(\cm)/\per(\cp)$ is an algebraic triangulated category, it follows that there is a dg $k$-algebra $B$ together with a triangle equivalence $F\colon \per(\cm)/\per(\cp)\to \per(B)$ taking $\tilde{T}$ to $B$. As a consequence it follows by Lemma~\ref{lem:silting-reduction} that $H^n(B)=0$ for $n>0$, $H^0(B)\cong \End_\cc(T)$, and $H^n(B)\cong \Hom_\cc(T,\Sigma^{n}T)$ for $n<0$. This proves (a).

Next we show that the equivalence $F\colon \per(\cm)/\per(\cp)\to \per(B)$ restricts to an equivalence $\per_{\underline{\cm}}(\cm)\to \cd_{fd}(B)$. Then (b) and (c) follows.

For $M\in\cm$ and $X\in\per_{\underline{\cm}}(\cm)$, the space $\bigoplus_{n\in\mathbb{Z}}\Hom_{\per(\cm)/\per(\cp)}(M^\wedge,\Sigma^n X)$, being isomorphic to $\bigoplus_{n\in\mathbb{Z}}\Hom_{\per(\cm)}(M^\wedge,\Sigma^n X)$, is finite-dimensional. It follows that the space $\bigoplus_{n\in\mathbb{Z}}\Hom_{\per(B)}(B,\Sigma^n F(X))$ is finite-dimensional, implying that $F(X)\in\cd_{fd}(B)$. To show that the restriction is dense, take a simple module $S$ over $\underline{\cm}$. Then $\Hom_{\per(\cm)}(\tilde{T},\Sigma^n S)\cong \Hom_{\per(\cm)/\per(\cp)}(\tilde{T},\Sigma^n S)$ is $1$-dimensional for $n=0$ and is trivial for $n\neq 0$. Therefore $\Hom_{\per(B)}(B,\Sigma^n F(S))$ is $1$-dimensional for $n=0$ and is trivial for $n\neq 0$. So $F(S)$ is a simple module over $H^0(B)$. But $\underline{\cm}$ is equivalent to $\proj H^0(B)$, implying that $F$ takes a complete set of pairwise non-isomorphic simple modules over $\underline{\cm}$ to a complete set of pairwise non-isomorphic simple modules over $H^0(B)$. Now $\cd_{fd}(B)$ is classically generated by simple $H^0(B)$-modules. It follows that $F$ restricts to an equivalence $\per_{\underline{\cm}}(\cm)\to\cd_{fd}(B)$.

Finally it follows by \cite[Proposition 5.4]{KellerReiten07} that there is a bifunctorial isomorphism for $X\in\mod\underline{\cm}$ and $Y\in\per_{\underline{\cm}}(\cm)$
\[
D\Hom(X,Y)\cong \Hom(Y,\Sigma^{d+1}X).
\]
We obtain (d) by applying the equivalence $F$.
\end{proof}

\section{Calabi--Yau categories with
cluster-tilting objects}
Let $k$ be a field. In this section we will apply our previous results
to study the canonical form of Calabi--Yau triangulated categories with cluster-tilting objects. 

\subsection{Ginzburg dg algebras}\label{ss:cy-alg}

\newcommand{\cten}{{\widehat{\ten}}}

Let $A$ be a pseudo-compact dg $k$-algebra, see~\cite{KellerYang11,VandenBergh15}. Let $A^e=A^{op}\cten A$
be the enveloping algebra, \ie the completion of $A^{op}\ten A$ with respect to the topology induced from $A$.
The dg algebra $A$ is \emph{topologically homologically smooth} if $A\in\per(A^e)$,
and is \emph{bimodule $d$-Calabi--Yau} if in addition there is an isomorphism $\eta\colon\RHom_{A^e}(A,A^e)\stackrel{\cong}{\longrightarrow}\Sigma^d A$ in $\cd(A^e)$. In the original definition of Ginzburg, $\eta$ was assumed to be self-dual, but this turns out to be automatic, see \cite[Appendix 14]{VandenBergh15}.

\begin{lemma}\label{l:bimodule-cy->triangulated-cy}
 Let $A$ be a bimodule $d$-Calabi--Yau pseudo-compact dg algebra. Then $\cd_{fd}(A)$ is $d$-Calabi--Yau as a triangulated category.
\end{lemma}
\begin{proof}
This immediately follows from~\cite[Lemma A.16]{KellerYang11}.
\end{proof}

Nice examples of bimodule Calabi--Yau dg algebras include complete Ginzburg dg algebras of quivers with potential. Let $Q$ be a finite quiver and $W$ be a formal combination of cycles of $Q$. The pair $(Q,W)$ is called a \emph{quiver with
potential}. From $Q$ we define a graded quiver $\tilde{Q}$, which has the same vertices as $Q$ and
whose arrows are
\begin{itemize}
\item[--] the arrows of $Q$ (they all have degree~$0$),
\item[--] an arrow $a^*: j \to i$ of degree $-1$ for each arrow $a:i\to j$ of $Q$,
\item[--] a loop $t_i : i \to i$ of degree $-2$ for each vertex $i$
of $Q$.
\end{itemize}
The \emph{complete Ginzburg dg algebra} $\widehat{\Gamma}(Q,W)$, introduced by Ginzburg in~\cite{Ginzburg06}, is the dg algebra
whose underlying graded algebra is the complete path algebra $\widehat{k\tilde{Q}}$ and whose differential is
the unique continuous linear differential which satisfies the graded Leibniz rule and which takes the following value
on the arrows of $\tilde{Q}$:
\begin{itemize}
\item[--] $d(a)=0$ for each arrow $a$ of $Q$,
\item[--] $d(a^*) = \del_a W$ for each arrow $a$ of $Q$,
\item[--] $d(t_i) = e_i (\sum_{a} (aa^*-a^*a)) e_i$ for each vertex $i$ of $Q$, where
$e_i$ is the trivial path at $i$ and the sum runs over the set of
arrows of $Q$.
\end{itemize}
Here for an arrow $a$ of $Q$, the \emph{cyclic derivative} $\del_a$
is the unique continuous linear map which takes a cycle $c$ to the sum
$ \sum_{c=u a v} vu $ taken over all decompositions of the cycle $c$
(where $u$ and $v$ are possibly trivial paths).

The \emph{complete Jacobian algebra} $\widehat{J}(Q,W)$ is the $0$-th cohomology of $\widehat{\Gamma}(Q,W)$. More precisely,
\[\widehat{J}(Q,W)=\widehat{kQ}/\overline{(\del_a W\mid a\in Q_1)}.\]

\begin{theorem}
\emph{(\cite[Theorem A.17]{KellerYang11},~\cite[Theorem 6.3]{Keller11})}  The dg algebra $\widehat{\Gamma}(Q,W)$ is topologically
homologically smooth and bimodule $3$-Calabi--Yau.
\end{theorem}

\subsection{Cluster categories}\label{ss:generalised-cluster-cat}

Let $A$ be a pseudo-compact dg $k$-algebra such that
\begin{itemize}
 \item[--] $A$ is non-positive,
 \item[--] $A$ is topologically homologically smooth,
 \item[--] $A$ is bimodule $3$-Calabi--Yau,
 \item[--] $H^0(A)$ is finite-dimensional.
\end{itemize}
Then $\cd_{fd}(A)\subseteq\per(A)$. Set
\[\cc_{A}:=\per(A)/\cd_{fd}(A),\]
and call it the \emph{cluster category} of $A$.

\begin{theorem}\label{t:generalized-cluster-cat}
\label{thm:Amiot}
 \emph{(\cite[Theorem 2.1]{Amiot09},~\cite[Theorem A.21]{KellerYang11},~\cite[Theorem 2.2]{Guolingyan11a},~\cite[Theorem 5.8]{IyamaYang18})}
The category $\cc_A$ is $2$-Calabi--Yau, and the image of $A$ under the canonical projection
$\per(A)\rightarrow\cc_A$ is a cluster-tilting object whose endomorphism algebra
is $H^0(A)$.
\end{theorem}

For example, for a quiver with potential $(Q,W)$ such that the complete Jacobian algebra $\widehat{J}(Q,W)$ is finite-dimensional,
the cluster category
\[\cc_{(Q,W)}:=\cc_{\widehat{\Gamma}(Q,W)}\]
is $2$-Calabi--Yau, and the image of $\widehat{\Gamma}(Q,W)$ in $\cc_{(Q,W)}$ is a (2-)cluster-tilting object
whose endomorphism algebra is $\widehat{J}(Q,W)$.

We will call the image of $A$ in $\cc_A$ the \emph{standard cluster-tilting object}.

\subsection{Calabi--Yau categories with cluster-tilting objects}

We propose the following conjecture (\confer \cite[Summary of results, Part 2, Perspectives]{Amiot08})..

\begin{conjecture}\label{conj:cy-cat-is-generalised-cluster-cat}
Let $k$ be algebraically closed of characteristic zero. Let $\cc$ be a $2$-Calabi--Yau algebraic triangulated $k$-category with a cluster-tilting object $T$.
Then $\cc$ is triangle equivalent to the
cluster category of some quiver with potential,
with the equivalence sending $T$ to the standard cluster-tilting object.
\end{conjecture}

Let $B$ be as in Theorem~\ref{thm:cluster-resolution}. Then $B$ is non-positive and topologically homologically smooth such that $\cc$ is triangle equivalent to $\per(B)/\cd_{fd}(B)$. If $B$ is quasi-equivalent to an exact $3$-bimodule dg algebra in the sense \cite[Section 1, Definition]{VandenBergh15}, then by \cite[Theorem B]{VandenBergh15}, $B$ is quasi-equivalent to the complete Ginzburg dg algebra of some quiver with potential $(Q,W)$, and therefore $\cc$ is triangle equivalent to the cluster category $\cc_{(Q,W)}$. Therefore we propose the following question. If it has a positive answer, then Conjecture~\ref{conj:cy-cat-is-generalised-cluster-cat} holds true. Note that  (ii) implies
(i) by Lemma~\ref{l:bimodule-cy->triangulated-cy}.

\begin{question} \label{Q:CYpropDG}
Let $A$ be a non-positive pseudo-compact topologically homologically smooth dg $k$-algebra
and let $K$ be the direct product of a finite copies of $k$.
Assume that
\begin{itemize}
\item[(a)] there is an injective homomorphism $\eta:K\rightarrow A$ and a surjective
homomorphism $\varepsilon:A\rightarrow K$ of dg algebras such that $\varepsilon\circ\eta=id_K$,
\item[(b)] $\cd_{fd}(A)=\thick(K)$, where $K$ is viewed as a dg $A$-module via $\varepsilon$,
\item[(c)] $H^0(A)$ is finite-dimensional over $k$.
\end{itemize}
Are the following conditions equivalent?
\begin{itemize}
 \item[(i)] \footnote{We point out that this condition is weaker than the one claimed in \cite[Section 7.2]{KalckYang16}.} there is a bifunctorial isomorphism for $X\in\mod H^0(A)$ and $Y\in\cd_{fd}(A)$
\[
D\Hom_{\cd_{fd}(A)}(X,Y)\cong \Hom_{\cd_{fd}(A)}(Y,\Sigma^3X),
\]
 \item[(ii)] $A$ is quasi-equivalent to an exact $3$-Calabi--Yau dg algebra.
\end{itemize}
\end{question}


In \cite{Amiot11}, Amiot asked the following question. If Conjecture~\ref{conj:cy-cat-is-generalised-cluster-cat} holds true, then this question has a positive answer.

\begin{question}\label{q:amiot-question}\emph{(\cite[Question 2.20.1]{Amiot11})}
Let $k$ be algebraically closed of characteristic zero. 
Let $\cc$ be a 2-Calabi--Yau algebraic triangulated category with a cluster-tilting object $T$.
Is $\End_{\cc}(T)$ isomorphic to the complete Jacobian algebra of some quiver with potential?
\end{question}

Note that  we have replaced `Jacobian algebra' by `complete Jacobian algebra'. The original question has a negative answer, as there are quivers with potentials whose complete Jacobian algebras are not non-complete Jacobian algebras of any quiver with potential, see \cite[Example 4.3]{Plamondon13} for an example.

Moreover, we have put an extra assumption on the characteristic of the field, 
as when $k$ is of positive characteristic, the answer to the question is negative. For example, if $k$ is of characteristic $p>0$, then $k[x]/(x^{p-1})$ is not a Jacobian algebra. However, take $\Gamma$ as the dg algebra whose underlying graded algebra is $k\langle\hspace{-3pt}\langle x,x^*,t\rangle\hspace{-3pt}\rangle$ with $\deg(x)=0$, $\deg(x^*)=-1$ and
$\deg(t)=-2$, and whose differential $d$
is defined by
\[
d(x)=0,~~~ d(x^*)= x^{p-1},~~~ d(t)= xx^*-x^*x.
\]
It is straightforward to check that $\Gamma$ satisfies the assumptions of Theorem~\ref{thm:Amiot}, so $k[x]/(x^{p-1})=H^0\Gamma$ is the endomorphism of a cluster-tilting object in a 2-Calabi--Yau algebraic triangulated category. More examples can be found in  \cite{Ladkani14}.

\subsection{Keller--Reiten's recognition theorem}

Let $Q$ be an acyclic quiver and $d\geq 2$. Define the $d$-cluster category of $Q$ as the orbit category
\[
\cc^{(d)}_Q:=\cd^b(\mod kQ)/\tau^{-1}\circ \Sigma^{d-1}.
\]
It is a $d$-Calabi--Yau triangulated category (\cite[Section 8]{Keller05} and \cite{Keller05cor}), and the image $T$ of $kQ$ is a $d$-cluster-tilting object with endomorphism algebra $\End_{\cc_Q}(T)=kQ$ (\cite[Proposition 1.7(d) and Theorem 3.3(b)]{BuanMarshReinekeReitenTodorov06},  \cite[Proposition 5.6]{KellerReiten07}). Moreover, $\Hom_{\cc_Q}(T)=0$ for $-d+2\leq p\leq -1$ (\cite[Lemma 4.1]{KellerReiten08}).

\begin{theorem}[{\cite[Theorems 2.1 and 4.2]{KellerReiten08}}]
\label{thm:KR's-recognition-theorem}
Let $\cc$ be a $d$-Calabi--Yau algebraic triangulated category. Assume that $T$ is a $d$-cluster-tilting object of $\cc$ such that $\Hom_\cc(T,\Sigma^p T)=0$ for $-d+2\leq p\leq -1$ and that $\End_\cc(T)\cong kQ$ for some finite acyclic quiver $Q$. Then $\cc$ is  triangle equivalent to $\cc^{(d)}_Q$.
\end{theorem}

We give a proof of this theorem under the extra conditions that $k$ is algebraically closed and  that $Q$ is a tree quiver, that is, there are no cycles in the underlying graph of $Q$.

\begin{proof}[Proof of Theorem~\ref{thm:KR's-recognition-theorem}]
Assume that $k$ is algebraically closed and that $Q$ is a tree quiver with vertices $\{1,\ldots,r\}$.

Let $B$ be the dg algebra obtained in Theorem~\ref{thm:cluster-resolution}. Then 
 $\cc$ is triangle equivalent to $\per(B)/\cd_{fd}(B)$. We will show that up to quasi-equivalence $B$ depends not on $\cc$ but on $Q$ only. As $\cc_Q^{(d)}$ satisfies all the assumptions, it is also triangle equivalent to $\per(B)/\cd_{fd}(B)$. Therefore $\cc$ is triangle equivalent to $\cc_Q^{(d)}$.

By Theorem~\ref{thm:cluster-resolution}, $B$ is non-positive and satisfies the three conditions (FD), (BE) and (HS) in Section~\ref{ss:augmentation}. Thus by Theorem~\ref{t:augmentation-of-non-pos-dg-alg}, we may assume that $B$ is the dual bar construction of its $A_\infty$-Koszul dual $B^*$, and therefore $B=(\widehat{kQ'},d)$ is a minimal complete dg-quiver algebra. By Theorem~\ref{thm:cluster-resolution}(a), $H^0(B)\cong kQ$ is hereditary, and $H^p(B)=0$ for all $-d+2\leq p\leq -1$. So applying Lemma~\ref{lem:dg-quiver-algebra-with-hereditary-0-th-cohomology} we know that $Q'$ has no arrows in degrees $-d+2\leq p\leq -1$. This implies that $(B^*)^p=0$ for $2\leq p\leq d-1$.

As a graded vector space $B^*=\bigoplus_{p\in\mathbb{Z}}\bigoplus_{i,j=1}^r\Hom_{\cd(B)}(S_i,\Sigma^p S_j)$. By Theorem~\ref{thm:cluster-resolution}(d) there is an isomorphism
\begin{align}\label{eq:cy-duality}
D\Hom(S_i,\Sigma^p S_j)\cong \Hom(S_j,\Sigma^{d+1-p}S_i).
\end{align}
Since $(B^*)^p=0$ for $p<0$, this implies that $(B^*)^p=0$ for $p\neq 0,1,d,d+1$.
Let $e_i=\id_{S_i}$. Then $K:=(B^*)^0$ has a basis $\{e_1,\ldots,e_r\}$. Let $\{e_1^*,\ldots,e_r^*\}$ be the dual basis in $(B^*)^{d+1}$. Let $Q^{gr,op}$ be the opposite quiver of $Q$ with all arrows in degree $1$. Since $H^0(B)\cong kQ$, it follows from the dual bar construction that $(B^*)^1$ has a basis $Q^{gr,op}_1$. Let $\{\alpha^*\mid \alpha\in Q^{gr,op}_1\}$ be the dual basis of $Q^{gr,op}_1$ in $(B^*)^{d}$.
Because the isomorphism \eqref{eq:cy-duality} is compatible with compositions (see for example \cite[Section 2]{HocheneggerKalckPloog16}), we have $\alpha\alpha^*\neq 0$ and $\alpha^*\alpha\neq 0$ for any $\alpha\in Q^{gr,op}_1$. Therefore $(B^*,m_2)=A_d(Q^{gr,op},\lambda,\mu)$ for some functions $\lambda,\mu\colon Q^{gr,op}_1\to k^\times$. So $B^*\cong A_d(Q^{gr,op},\mathbf{1},\mathbf{1})$ by Lemma~\ref{lem:formality-of-trivial-extension} and Lemma~\ref{lem:no-deformation-trivial-extension} and thus depends on $Q$ only. So does $B$.
\end{proof}

\def\cprime{$'$} \def\cprime{$'$}
\providecommand{\bysame}{\leavevmode\hbox to3em{\hrulefill}\thinspace}
\providecommand{\MR}{\relax\ifhmode\unskip\space\fi MR }
\providecommand{\MRhref}[2]{%
  \href{http://www.ams.org/mathscinet-getitem?mr=#1}{#2}
}
\providecommand{\href}[2]{#2}

\end{document}